\documentclass[a4paper,english,reqno]{amsart}
\usepackage[T1]{fontenc}
\usepackage[latin1]{inputenc}
\usepackage{xcolor}
\usepackage{pdfcolmk}
\usepackage{amsthm}
\usepackage{amstext}
\usepackage{graphicx}
\usepackage{setspace}
\usepackage{amssymb}
\usepackage{esint}
\PassOptionsToPackage{normalem}{ulem}
\usepackage{ulem}
\makeatletter

\providecolor{lyxadded}{rgb}{0,0,1}
\providecolor{lyxdeleted}{rgb}{1,0,0}

\numberwithin{equation}{section} 
\numberwithin{figure}{section} 
\theoremstyle{plain}
\newtheorem{thm}{Theorem}[section]
  \theoremstyle{remark}
  \newtheorem{rem}[thm]{Remark}
  \theoremstyle{plain}
  \newtheorem{fact}[thm]{Fact}
  \theoremstyle{definition}
  \newtheorem*{example*}{Example}
  \theoremstyle{plain}
  \newtheorem{lem}[thm]{Lemma}
  \theoremstyle{definition}
  \newtheorem{defn}[thm]{Definition}
  \theoremstyle{plain}
  \newtheorem{prop}[thm]{Proposition}
  \theoremstyle{plain}
  \newtheorem{assumption}[thm]{Assumption}
 \theoremstyle{definition}
  \newtheorem{example}[thm]{Example}

\usepackage{fancyhdr}
\usepackage{bbm}
\usepackage
{caption,subfig}

\DeclareMathOperator{\card}{card}

\DeclareMathOperator{\spn}{span}
\DeclareMathOperator{\cl}{cl}

\newcommand{\Z}{\mathbb{Z}}

\newcommand{\R}{\mathbb {R}}

\newcommand{\C}{\mathbb {C}}

\newcommand{\N}{\mathbb {N}}

\newcommand{\1}{\mathbbm{1}}

\renewcommand{\epsilon}{\varepsilon}
\renewcommand{\chi}{\1}
\renewcommand{\emptyset}{\varnothing}

\@ifundefined{showcaptionsetup}{}{%
 \PassOptionsToPackage{caption=false}{subfig}}
\usepackage{subfig}
\makeatother

\usepackage{babel}

\begin{document}
\selectlanguage{english}

\title{Wavelets for iterated function systems}

\author{Jana Bohnstengel and Marc Kesseböhmer}

\date{\today}

\address{Universität Bremen, Fachbereich 3 - Mathematik und Informatik, Bibliothekstrasse
1, 28349 Bremen, Germany}
\begin{abstract}
We construct a wavelet and a generalised Fourier basis with respect
to some fractal measures given by one-dimensional iterated function
systems. In this paper we will not assume that these systems are given
by linear contractions generalising in this way some previous work
of Jørgensen and Dutkay to the non-linear setting. 
\end{abstract}



\maketitle

\section{Introduction}

In this paper we will construct wavelets and generalised Fourier bases
on fractal sets constructed via iterated function systems (IFS), which
we do not assume to be linear. More precisely, the wavelets under
considerations are constructed in the $L^{2}$-space with respect
to the measure of maximal entropy transported to the so-called enlarged
fractal, which is dense in $\mathbb{R}$.

It is a natural approach to consider wavelets in the context of such
fractals since both carry a self-similar structure; the fractal inherits
it from the prescribed scaling of the IFS while the wavelet satisfies
a certain scaling identity (see e.g. (\ref{eq:ScalingEq})). Another
interesting aspect is that both wavelets and fractals are used in
image compression, where both have advantages and disadvantages like
blurring by zooming in or long compression times. Because of these
common features it is of interest to develop a common mathematical
foundation of these objects not least to find out whether it can have
an impact on the theory of data or image compression.

The aim of the wavelet analysis is to approximate functions by using
superposition from a wavelet basis. This basis is supposed to be orthonormal
and derived from a finite set of functions, the so-called mother wavelets
(cf. Proposition \ref{pro:motherW}). To obtain such a basis we employ
the multiresolution analysis (MRA) (cf. Definition \ref{def: MRA-1}).
Our main goal is therefore to set up a MRA in the non-linear situation.
For this we generalise some ideas from \cite{waveletsonfractals,JorgensenPedersen},
which are restricted to homogeneous linear cases with respect to the
restriction of certain Hausdorff measures. 

In Section \ref{sec:Fourier-Basis} we are also going to generalise
the construction of the Fourier basis in the sense of \cite{waveletsonfractals}
to our non-linear setting. This will be done in virtue of a homeomorphism
conjugating the IFS under consideration with a linear homogeneous
IFS. As a consequent of this construction we are able to set up a
generalised Fourier basis also for such linear IFS which do not a
allow a Fourier basis in the sense of \cite{waveletsonfractals} (cf.
Example \ref{exa:-1/3-Cantor-set}).

\section{Wavelets }

\subsection{Enlarged fractal and the measure of maximal entropy\label{homeophi}}

The family \[
\mathcal{S}:=\left(\sigma_{i}:\left[0,1\right]\to\left[0,1\right]:i\in\underline{p}:=\left\{ 0,\ldots,p-1\right\} \right)\]
consisting of $p\in\N$ injective contractions $\sigma_{i}$, which
are uniformly Lipschitz with Lipschitz-constant $0<c_{\mathcal{S}}<1$,
i.e. $|\sigma_{i}(x)-\sigma_{i}(y)|\leq c_{\mathcal{S}}|x-y|$, $x,y\in[0,1]$,
$i\in\underline{p}$. We will always assume that all contractions
have the same orientation (in fact are increasing) and that the IFS
satisfies the open set condition (OSC), i.e. $\bigcup_{i=0}^{p-1}\sigma_{i}((0,1))\subset(0,1)$
and $\sigma_{i}((0,1))\cap\sigma_{j}((0,1))=\emptyset$, $i\neq j$.

It is well known that there exists a unique non empty compact set
$C\subset[0,1]$ such that $C=\bigcup_{i=0}^{p-1}\sigma_{i}(C)$.
This set will be denoted the \emph{limit set} of $\mathcal{S}$. Throughout,
we will assume that the IFS $\left(\sigma_{i}\right)_{i=0}^{p-1}$
is arrange in ascending order, that is $\sigma_{i}([0,1])$ lies to
the left of $\sigma_{i+1}([0,1])$ for all $i=0,\dots,p-2$.

It is always possible to extend the IFS $\mathcal{S}$ by linear contractions
to obtain the IFS \[
\mathcal{T}=\left(\tau_{i}:\left[0,1\right]\to\left[0,1\right]:i\in\underline{N}:=\left\{ 0,\ldots,N-1\right\} \right)\]
which leaves no gaps. More precisely, there exists a number $N\geq p$
and a set $A\subset\underline{N}$ such that 
\begin{enumerate}
\item $\left\{ \tau_{j}:j\in A\right\} =\left\{ \sigma_{i}:i\in\underline{p}\right\} $,
\item $\tau_{0}\left(0\right)=0$, $\tau_{N-1}\left(1\right)=1$ and $\tau_{i}\left(1\right)=\tau_{i+1}\left(0\right)$,
$i=1,\ldots,N-2$,
\item $\forall i\in\underline{N}\setminus A$: $\tau_{i}:\left[0,1\right]\to\left[0,1\right]$
is an affine increasing contraction.
\end{enumerate}
In the following the uniform Lipschitz constant for the IFS $\mathcal{T}$
will be denoted by $c_{\mathcal{T}}$.
\begin{rem}
Note that it is not essential to choose the {}``gap filling functions''
$\tau_{i}$, $i\in\underline{N}\setminus A$, to be affine. Our analysis
would work for any set of contracting injections as long as (1), (2)
and (3) above are satisfied. Nevertheless, the particular choice has
an influence on the the set $R$ and the measure $H$ defined below.
Also note that more than one gap filling function can be defined on
one gap. Throughout, let \[
\rho_{j,N}:\, x\mapsto\frac{x+j}{N}.\]
Then for instance, if $\mathcal{S}$ consists of functions $\sigma_{i}(x)=\frac{x+a_{i}}{N}$,
$a_{i}\in\underline{N}$, $i\in\underline{p},$ it is a natural choice
to extend $\mathcal{S}$ by the functions $\tau_{i}(x)=\rho_{i,N}(x)$,
$i\in\underline{N}\backslash A$, such that $\mathcal{T}$ is equal
to $\left\{ \rho_{i,N}:i\in\underline{N}\right\} $.
\end{rem}
For $\omega:=(i_{1},\dots,i_{k})\in\underline{N}^{k}$ let $\tau_{\omega}:=\tau_{i_{k}}\circ\dots\circ\tau_{i_{1}}$
and $\tau_{\emptyset}=\mbox{id}$ be the identity on $\left[0,1\right]$.
Next we define the\emph{ enlarged fractal} $R$ in two steps. First
we fill the gaps of the fractal $C$ with scaled copies of itself
by letting \[
R_{[0,1]}:=\bigcup_{k\geq0}\bigcup_{\omega\in\underline{N}^{k}}\tau_{\omega}(C),\]
 and then set \[
R:=R_{\left[0,1\right]}+\Z=\bigcup_{l\in\mathbb{Z}}R_{[0,1]}+l.\]
Now let \[
\Sigma:=\left\{ (i_{1},\dots,i_{k})\in\underline{N}^{k}:\: k\in\mathbb{N},\: i_{1}\notin A\right\} \]
be the set of finite words over the alphabet $\underline{N}$ such
that the initial letter is not from $A$. Then we can also write $R_{[0,1]}$
as the \emph{disjoint} union \[
R_{[0,1]}=\bigcup_{\omega\in\Sigma\cup\left\{ \emptyset\right\} }\tau_{\omega}(C).\]

\subsubsection{Fractal measures on the enlarged fractal\label{kolmo}}

In this section we will introduce the appropriate measure $H$ on
$\R$ needed for the MRA. The measure will be first defined on $[0,1]$
and then on $\R$. The construction is analogue to the construction
of $R_{[0,1]}$ and $R$. Let $\mu$ be the self-similar Borel probability
measure supported on $C$ associated to $\mathcal{S}$ with constant
weights, i.e. the unique probability measure $\mu$ satisfying $\mu=\frac{1}{p}\sum_{i\in A}\mu\circ\tau_{i}^{-1}$.
This measure has the property that each set of the form $\tau_{\omega}(C)$,
$\omega\in\underline{N}^{k}$, has measure $p^{-k}$. Thus, $\mu$
is the measure of maximal entropy in the sense of a shift dynamical
system. 

For $\omega=(i_{1},\dots,i_{k})\in\underline{N}^{k}$ we let $\left|\omega\right|=k$
denote the length of $\omega$ and $\left|\emptyset\right|=0$. 
\begin{fact}
The function $\nu:\mathcal{B}\rightarrow\overline{\mathbb{R}}_{0}^{+}$
given by \[
\nu:=\sum_{\omega\in\Sigma\cup\left\{ \emptyset\right\} }p^{-\left|\omega\right|}\mu\circ\tau_{\omega}^{-1}\]
defines a measure on $\left[0,1\right]$. Also, the sum of its translates
\[
H:\mathcal{B}\to\mathbb{R}_{0}^{+},\: B\mapsto\sum_{k\in\mathbb{Z}}\nu(B+k),\]
defines a measure. Its essential support is equal to $R$.
\end{fact}
Throughout, for $x\in\R$, let\[
\sigma(x):=\sum_{k\in\Z}\1_{[k,k+1)}\left(x\right)\left(Nk+\sum_{i=0}^{N-1}\1_{[\tau_{i}(0),\tau_{i}(1))}(x-k)\left(\tau_{i}^{-1}(x-k)+i\right)\right)\]
denote the \emph{scaling function} associated to $\mathcal{T}$. Note
that $\sigma^{-1}$ is given, for $x\in\R$, by \[
\sigma^{-1}(x)=\sum_{k\in\Z}\1_{[Nk,N(k+1))}\left(x\right)\left(k+\sum_{i=0}^{N-1}\1_{[i,i+1)}(x-Nk)\left(\tau_{i}(x-i-Nk)\right)\right).\]

\begin{example*}
Let us give an example for a scaling function $\sigma$.

Take the IFS $\mathcal{S=}\left(\sigma_{0},\sigma_{1}\right)$ and
the extended IFS $\mathcal{T}=\left(\tau_{0},\tau_{1},\tau_{2}\right)$
with $\sigma_{0}=\tau_{0}:\, x\mapsto\frac{x^{2}}{5}+\frac{2}{5}x$,
$\tau_{1}:\, x\mapsto\frac{x}{5}+\frac{3}{5}$ and $\sigma_{1}=\tau_{2}:\, x\mapsto\frac{\log(x+1)}{5\cdot\log(2)}+\frac{4}{5}$.
The inverse branches of the maps are illustrated in Figure \ref{fig:IFS}\subref{branches IFS}.
The scaling function $\sigma$ and the inverse of the scaling function
$\sigma^{-1}$ are then given by:\begin{align*}
\sigma(x)= & \sum_{k\in\Z}\1_{[k,k+1)}\left(x\right)\Bigg(\1_{[0,3/5)}(x-k)\left(\sqrt{5(x-k)+1}-1\right)\\
 & +\1_{[3/5,4/5)}(x-k)\left(5(x-k)-2\right)+\1_{[4/5,1)}(x-k)\left(2^{5(x-k)-4}+1\right)+3k\Bigg),\end{align*}
\begin{align*}
\sigma^{-1}(x)= & \sum_{k\in\Z}\1_{[3k,3(k+1))}\left(x\right)\Bigg(\1_{[0,1)}(x-3k)\left(\frac{(x-3k)^{2}}{5}+\frac{2}{5}x-\frac{1}{5}k\right)\\
 & +\1_{[1,2)}(x-3k)\left(\frac{x}{5}+\frac{2}{5}+\frac{2}{5}k\right)+\1_{[2,3)}(x-3k)\left(\frac{\log(x-3k-1)}{5\log(2)}+\frac{4}{5}+k\right)\Bigg).\end{align*}

The graph of $\sigma$ is shown in Figure \ref{fig:IFS}\subref{scaling function}.\renewcommand{\thesubfigure}{ {\rm (\alph{subfigure})}}
\captionsetup[subfigure]{labelformat=simple, labelsep=space, listofformat=subsimple}

\begin{figure}
\hfill{}\subfloat[\label{branches IFS}]{\includegraphics[height=5cm]{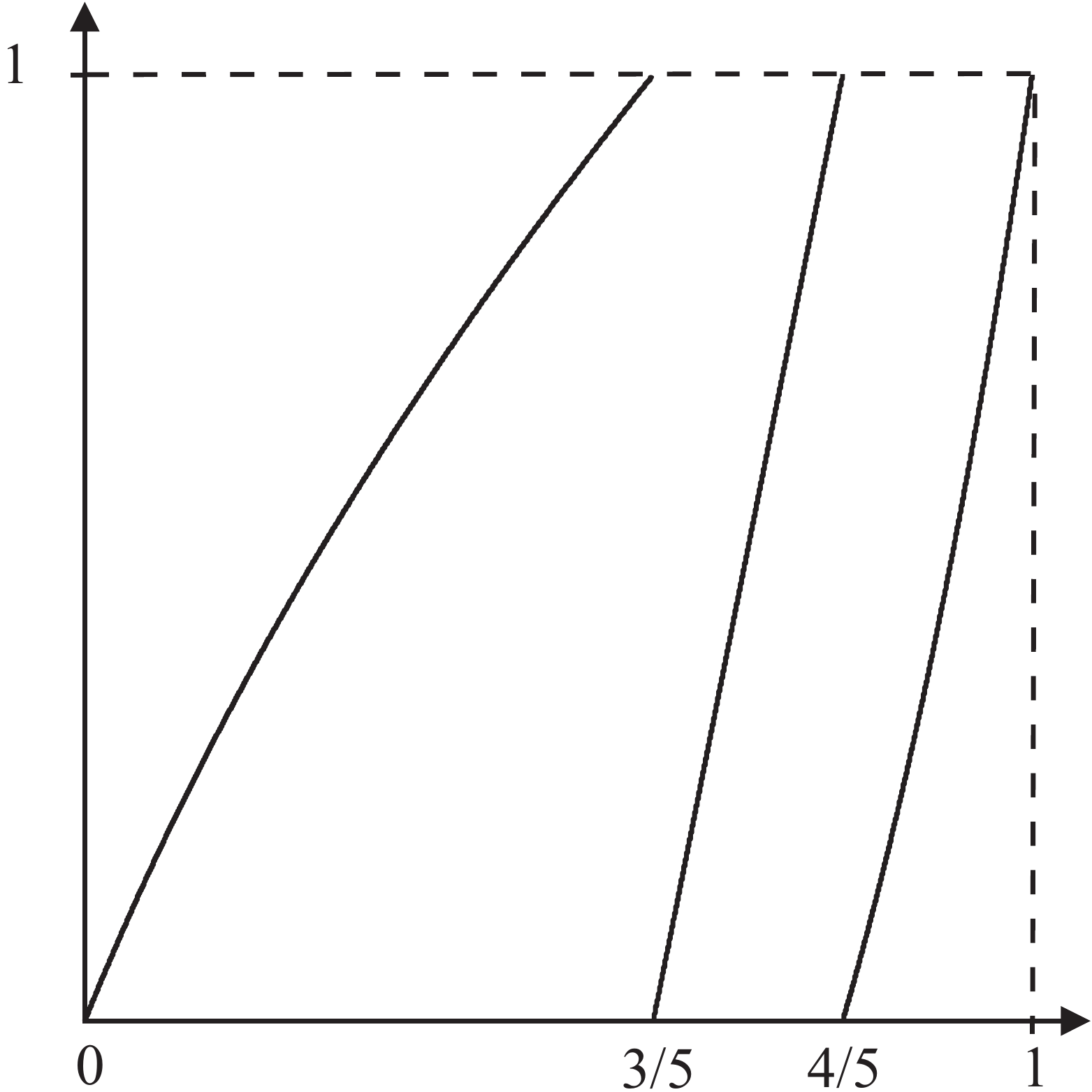}}\hspace{3cm}\subfloat[\label{scaling function}]{\includegraphics[height=5cm]{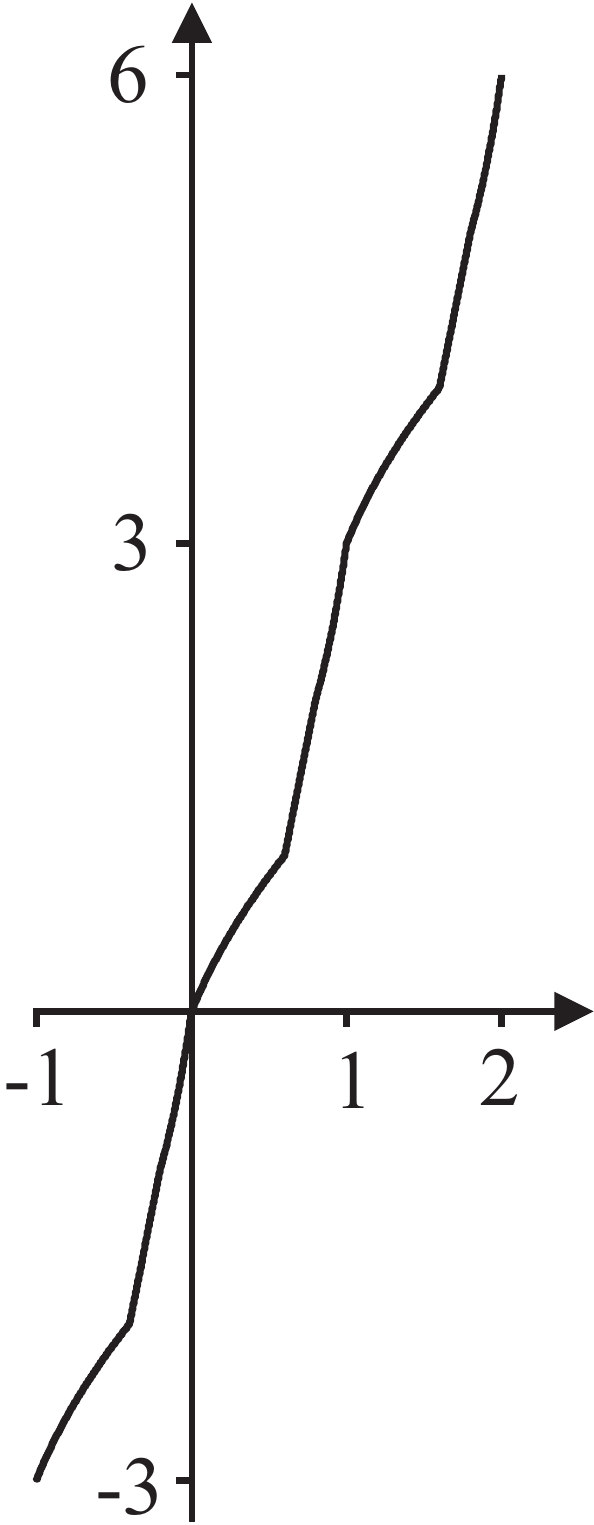}

}\hfill{}

\caption{\label{fig:IFS}The expanding inverse branches of an IFS with corresponding
scaling function $\sigma$. }

\end{figure}

\end{example*}
Let us now turn back to the measure $H$.
\begin{lem}
\label{ma=0000DFeigenschaft}We have $H\circ\sigma=pH$ and in particular,
for all $i\in\underline{N}$, $\nu\circ\tau_{i}=p^{-1}\nu$.\end{lem}
\begin{proof}
For $E\in\mathcal{B}$ we have\begin{align*}
H(\sigma(E)) & =\sum_{k\in\Z}\nu(\sigma(E)+k)=\sum_{i=0}^{N-1}\sum_{l\in\Z}\nu(\sigma(E)-Nl-i)\\
 & =\sum_{l\in\Z}\nu\left(\bigcup_{i=0}^{N-1}\tau_{i}^{-1}(E-l)\right).\end{align*}
 Since \begin{align*}
 & \!\!\!\!\!\!\!\!\!\!\nu\left(\bigcup_{i=0}^{N-1}\tau_{i}^{-1}(E-l)\right)\\
 & =\sum_{\omega\in\Sigma\cup\left\{ \emptyset\right\} }p^{-|\omega|}\mu\left(\tau_{\omega}^{-1}\left(\bigcup_{i=0}^{N-1}\tau_{i}^{-1}(E-l)\right)\right)\\
 & =\sum_{i=0}^{N-1}\sum_{\omega\in\Sigma\cup\left\{ \emptyset\right\} }p^{-|\omega|}\mu(\tau_{\omega}^{-1}(\tau_{i}^{-1}(E-l)))\\
 & =\sum_{i=0}^{N-1}\sum_{\omega\in\Sigma}p^{-|\omega|}\mu(\tau_{\omega i}^{-1}(E-l))+\sum_{i=0}^{N-1}\mu(\tau_{i}^{-1}(E-l))\\
 & =p\sum_{\omega\in\Sigma^{*}}p^{-|\omega|}\mu(\tau_{\omega}^{-1}(E-l))+\sum_{i\notin A}\mu(\tau_{i}^{-1}(E-l))+\sum_{i\in A}\mu(\tau_{i}^{-1}(E-l))\\
 & =p\sum_{\omega\in\Sigma}p^{-|\omega|}\mu(\tau_{\omega}^{-1}(E-l))+p\mu(E-l)\\
 & =p\nu(E-l),\end{align*}
where $\Sigma^{*}=\left\{ (i_{1},\dots,i_{k})\in\underline{N}^{k}:\: k\geq2,\: i_{1}\notin A\right\} $,
we have $H(\sigma(E))=pH(E)$.
\end{proof}

\subsection{Construction of wavelet bases for general self-similar fractals\label{wavelet}}

In this section we will show how to find a wavelet basis for $L^{2}\left(H\right)$.
This wavelet basis is constructed via an MRA. In our context the definition
of the MRA is given as follows.
\begin{defn}
\label{def: MRA-1} Let $\sigma:\R\to\R$ be a continuous increasing
function, such that, for some fixed $N\in\N,$ \[
\sigma(x+k)=\sigma(x)+Nk,\quad x\in\left[0,1\right],\: k\in\Z.\]
Furthermore, let $H$ be a measure on $\left(\R,\mathcal{B}\right)$
such that $H\left(A\right)=H\left(A+k\right)$, $A\in\mathcal{B}$,
$k\in\Z$ and $H(\sigma(A))=pH(A)$, for some $p\in\N$. We say $\left(H,\sigma\right)$
allows a\emph{ multiresolution analysis (MRA)} if there exists a family
$\{V_{j}:j\in\mathbb{Z}\}$ of closed subspaces of $L^{2}\left(H\right)$
and a function $\varphi\in L^{2}\left(H\right)$ (called the \emph{father
wavelet}) such that the following conditions are satisfied. 
\begin{enumerate}
\item $\dots\subset V_{2}\subset V_{1}\subset V_{0}\subset V_{-1}\subset V_{-2}\subset\dots$,
\label{enu:DefMRA1}
\item $\text{cl}\bigcup_{j\in\mathbb{Z}}V_{j}=L^{2}(H)$, 
\item $\bigcap_{j\in\mathbb{Z}}V_{j}=\{0\}$, 
\item $f\in V_{j}\iff f\circ\sigma\in V_{j-1}$, $j\in\mathbb{Z}$,
\item $\{x\mapsto\varphi(x-n):\; n\in\mathbb{Z}\}$ is an orthonormal basis
in $V_{0}$.\label{enu:DefMRA5} 
\end{enumerate}
Note that for $\sigma:x\mapsto Nx$ and $H$ chosen to be the Lebesgue
measure, this definition coincides with the classical definition of
the MRA (see e.g. \cite{daubechies}). 

\end{defn}
Let us now define the \emph{shift operator} $T$ and the \emph{scaling
operator} $U$ on $L^{2}(H)$ by \[
(Tg)(x)=g(x-1)\;\mbox{and }\;(Ug)(x)=\frac{1}{\sqrt{p}}g(\sigma^{-1}(x)),\; g\in L^{2}(H),\: x\in\R.\]

\begin{rem}
Both operators $U$ and $T$ are unitary.
\end{rem}
For the remaining part of this section we will demonstrate that the
MRA can be satisfied if we choose the father wavelet $\varphi$ to
be the characteristic function of the fractal $C$, i.e. \[
\varphi:=\chi_{C}.\]
First we observe that the function $\varphi$ satisfies the following
scaling equation for $H$-almost every $x\in\R$.\begin{align}
\varphi(x) & =\chi_{\bigcup_{i\in A}\tau_{i}(C)}(x)=\sum_{i\in A}\chi_{\tau_{i}(C)}(x)=\sum_{i\in A}\chi_{C}(\tau_{i}^{-1}(x))=\sum_{i\in A}\varphi(\tau_{i}^{-1}(x))\nonumber \\
 & =\sum_{i\in A}\varphi(\sigma(x)-i).\label{eq:ScalingEq}\end{align}
By virtue of the so-called low-pass filter $m_{0}$ there exists a
relation between the two operators $T$ and $U.$ Let us take the
filter $m_{0}$ to be given by \[
m_{0}(z):=\frac{1}{\sqrt{p}}\sum_{i\in A}z^{i},\qquad z\in\mathbb{T}:=\left\{ z\in\C:|z|=1\right\} .\]
Note that $m_{0}$ can also be regarded as a function on the set of
unitary operators acting on $L^{2}\left(H\right)$. For this choice
the following proposition holds.
\begin{prop}
\label{ortho} The above defined operators $U$ and $T$ satisfy the
following relations.
\begin{enumerate}
\item $\left\langle T^{k}\varphi|T^{\ell}\varphi\right\rangle _{H}=\delta_{\ell k},$
$k,\ell\in\Z$, 
\item $U\varphi=m_{0}\left(T\right)\varphi,$ 
\item $UTU^{-1}=T^{N}.$
\end{enumerate}
\end{prop}
\begin{proof}
ad (1): Since $H\left((C+k)\cap(C+\ell)\right)=0$ for $k\not=\ell$
and $H$ is invariant with respect to the mapping $x\mapsto x+1$,
we have, \[
\begin{array}{lll}
\left\langle T^{k}\varphi|T^{\ell}\varphi\right\rangle _{H} & = & \int T^{k}\varphi(x)\overline{T^{\ell}\varphi(x)}\, dH(x)\\
 & = & \int\chi_{C}(x-k)\overline{\chi_{C}(x-\ell)}\, dH(x)\\
 & = & \int\chi_{C+k}(x)\chi_{C+\ell}(x)\, dH(x)=\delta_{kl}.\end{array}\]

ad (2): For $x\in\R$ we have \begin{align*}
U\varphi(x) & =\frac{1}{\sqrt{p}}\chi_{C}\left(\sigma^{-1}(x)\right)=\frac{1}{\sqrt{p}}\chi_{\sigma\left(C\right)}\left(x\right)=\frac{1}{\sqrt{p}}\chi_{\cup_{i\in A}\left(\tau_{i}^{-1}\left(\tau_{i}(C)\right)\right)+i}\left(x\right)\\
 & =\frac{1}{\sqrt{p}}\sum_{i\in A}\chi_{C}(x-i)=m_{0}(T)\varphi(x).\end{align*}

ad (3): Let $f\in L^{2}(H)$ and $x\in\R$. Then \[
\left(UTU^{-1}f\right)(x)=f\left(\sigma(\sigma^{-1}(x)-1)\right).\]
For $x\in[Nk,N(k+1))$ we have that $\sigma^{-1}(x)=\tau_{j}(x-j-Nk)+k$
for some $j\in\{0,\dots,N-1\}$. Thus, $\sigma^{-1}(x)\in[k,k+1)$
and $\sigma^{-1}(x)-1\in[k-1,k)$. Now observe\begin{align*}
\sigma(\sigma^{-1}(x)-1) & =\sigma\left(\tau_{j}(x-j-Nk)+k-1\right)\\
 & =\sum_{i=0}^{N-1}\1_{[\tau_{i}(0),\tau_{i}(1))}(\tau_{j}(x-j-Nk)+k-1-(k-1))\\
 & \qquad\qquad\times\left(\tau_{i}^{-1}\left(\tau_{j}(x-j-Nk)+k-1-(k-1)\right)+i\right)+N(k-1)\\
 & =\tau_{j}^{-1}\left(\tau_{j}(x-j-Nk)+k-1-(k-1)\right)+j+N(k-1)\\
 & =x-N.\end{align*}
 Consequently, $f\left(\sigma(\sigma^{-1}(x)-1)\right)=f(x-N)=T^{N}f(x)$.\end{proof}
\begin{rem}
Notice that \begin{align*}
U^{j}T^{k}\varphi(x) & =\left(\frac{1}{\sqrt{p}}\right)^{j}\varphi(\sigma^{-j}(x)-k)=\left(\frac{1}{\sqrt{p}}\right)^{j}\chi_{C}(\sigma^{-j}(x)-k)\\
 & =\left(\frac{1}{\sqrt{p}}\right)^{j}\chi_{\sigma^{j}(C+k)}(x).\end{align*}
 \end{rem}
\begin{thm}
\label{mra}The pair $(\sigma,H)$ allows an MRA if we set $\varphi:=\chi_{C}$
to be the father wavelet and let \textup{$V_{0}:=\cl\spn\left\{ T^{k}\varphi:k\in\Z\right\} $,
$V_{j}:=\cl\spn\left\{ U^{j}T^{k}\varphi:k\in\Z\right\} $, $j\in\Z$.}
In particular, we have \[
\cl\spn\left\{ U^{n}T^{k}\varphi:k,n\in\mathbb{Z}\right\} =L^{2}(H).\]
 \end{thm}
\begin{proof}
To prove that this gives an MRA, we show that the conditions (\ref{enu:DefMRA1})
to (\ref{enu:DefMRA5}) from Definition \ref{def: MRA-1} are satisfied. 

ad (1): Recall that $U\varphi=m_{0}(T)\varphi$ and $UTU^{-1}=T^{N}$.
Consequently, \[
U^{-1}T^{k}\varphi=U^{-1}m_{0}(T)U^{-1}T^{k}\varphi=U^{-2}m_{0}(T^{N})T^{k}\varphi.\]
This shows that $V_{-1}\subset V_{-2}$ and iterating this argument
it follows that $\dots V_{1}\subset V_{0}\subset V_{-1}\subset\dots$.

ad (3): Clearly $0\in\bigcap_{j\in\mathbb{Z}}V_{j}$. Recall that
$R$ is equal to the essential support of $H$. Now take $f\in\bigcap_{j\in\mathbb{Z}}V_{j}$.
Then $f\in V_{j}$ for all $j\in\Z$. Notice that if $0\neq f\in V_{j}$
for some $j_{0}\in\Z$ it follows that $f|_{\sigma^{j_{0}}(C+k)}=c\1_{\sigma^{j_{0}}(C+k)}$,
$c\neq0$ and since $\left(V_{j}\right)_{j\leq j_{0}}$ is a nested
sequence it follows that for every $j\leq j_{0}$ there exists exactly
one $k_{j}\in\Z$ such that $f|_{\sigma^{j}(C+k_{j})}=c\1_{\sigma^{j}(C+k_{j})}$
and consequently $f$ takes the value $c$ on the nested union $\bigcup_{j\leq j_{0}}\sigma^{j}(C+k_{j})$.
Since this union has infinite measure, $f$ must be constantly $0$.

ad (4): Let $f\in V_{j}$, i.e. for some $b_{k}\in\mathbb{C}$ \[
f(x)=\sum_{k}b_{k}U^{j}T^{k}\varphi(x)=\sum_{k}c_{k}\varphi(\sigma^{-j}(x)-k)\]
 and \[
\begin{array}{lllll}
f(\sigma(x)) & = & \sum_{k}b_{k}U^{j}T^{k}\varphi(\sigma(x)) & = & \sum_{k}c_{k}\varphi(\sigma^{-j}(\sigma(x))-k)\\
 & = & \sum_{k}c_{k}\varphi(\sigma^{-j+1}(x)-k) & = & \sum_{k}d_{k}U^{j-1}T^{k}\varphi(x).\end{array}\]
 Thus, $f\circ\sigma\in V_{j-1}$.

ad (5): In Proposition \ref{ortho} it has been shown that $\left(T^{k}\varphi\right)_{k\in\Z}$
is orthonormal. The spanning condition is trivially satisfied.

ad (2): First we will shown that $\1_{F}$ for $F\in\mathcal{B}|_{[k,k+1]}\,\mod\, H$,
$k\in\Z$, can be approximated by linear combinations of $U^{n}T^{m}\varphi=\sqrt{1/p}^{n}\cdot\1_{\sigma^{n}(C+m)}$,
$m,n\in\Z$. For this let us define the set \[
\mathcal{V}_{k}:=\left\{ C_{\omega,k}:=\sigma^{-n}\left(C+\sum_{i=0}^{n-1}a_{i}N^{i}+N^{n}k\right):\,\omega\in\mathcal{C}_{n},\, n\geq1\right\} ,\]
where\[
\mathcal{C}_{n}:=\left\{ (a_{1},\dots,a_{n}):\, a_{i}\in\underline{N}\right\} .\]
We are going to show that $\mathcal{V}_{k}\cup\emptyset$ defines
a semiring for $[k,k+1]$. Since \[
C+\sum_{i=0}^{n-1}a_{i}N^{i}+N^{n}k\subset\left[N\left(\sum_{i=1}^{n-1}a_{i}N^{i-1}+N^{n-1}k\right),N\left(\sum_{i=1}^{n-1}a_{i}N^{i-1}+N^{n-1}k+1\right)\right]\]
 we get inductively \[
\begin{array}{lll}
 & \!\!\!\!\!\!\!\!\!\!\!\!\sigma^{-n}\left(C+\sum_{i=0}^{n-1}a_{i}N^{i}+N^{n-1}k\right)\\
= & \sigma^{-n+1}\Bigg(\sum\1_{[j,j+1)}\left(C+\sum_{i=0}^{n-1}a_{i}N^{i}+N^{n}k-N\left(\sum_{i=1}^{n-1}a_{i}N^{i-1}+N^{n-1}k\right)\right)\\
 & \,\left(\tau_{j}\left(C+\sum_{i=0}^{n-1}a_{i}N^{i}+N^{n}k-N\left(\sum_{i=1}^{n-1}a_{i}N^{i-1}+N^{n-1}k\right)-j\right)+\sum_{i=1}^{n-1}a_{i}N^{i-1}+N^{n-1}k\right)\Bigg)\\
= & \sigma^{-n+1}\left(\sum_{j=0}^{N-1}\1_{[j,j+1)}(C+a_{0})\left(\tau_{j}(C+a_{0}-j)+\sum_{i=1}^{n-1}a_{i}N^{i-1}+N^{n-1}k\right)\right)\\
= & \sigma^{-n+1}\left(\tau_{a_{0}}(C)+\sum_{i=1}^{n-1}a_{i}N^{i-1}+N^{n-1}k\right)\\
= & \sigma^{-n+2}\left(\tau_{a_{1}}\left(\tau_{a_{0}}(C)\right)+\sum_{i=2}^{n-1}a_{i}N^{i-2}+N^{n-2}k\right)\\
\vdots\\
= & \sigma^{-1}\left(\tau_{a_{n-2}}\left(\cdots\left(\tau_{a_{1}}\left(\tau_{a_{0}}(C)\right)\right)\cdots\right)+a_{n-1}+Nk\right)\\
= & \tau_{a_{n-1}}\left(\tau_{a_{n-2}}\left(\cdots\left(\tau_{a_{1}}\left(\tau_{a_{0}}(C)\right)\right)\cdots\right)\right)+k.\end{array}\]
From this the semiring properties of $\mathcal{V}_{k}\cup\left\{ \emptyset\right\} $
follow immediately. Then also \[
\mathcal{V}:=\left\{ \bigcup_{\ell=1}^{m}B_{\ell}:B_{\ell}\in\bigcup_{k}\mathcal{V}_{k},m\in\N\right\} \]
defines a semiring. Furthermore, we will show that $\mathcal{B}|_{[k,k+1]}\subset\sigma\left(\mathcal{V}_{k}\right)$
which would also imply $\mathcal{B}\subset\sigma\left(\mathcal{V}\right)$.
In fact $\mathcal{B}|_{[k,k+1]}$ is generated $\!\!\mod H$ by sets
of the form $(a,b)\cap[k,k+1]\cap R$, $a,b\in\mathbb{R}$, which
belong obviously to $\sigma\left(\mathcal{V}_{k}\right)$. This shows
$\mathcal{B}|_{[k,k+1]}\subset\sigma\left(\mathcal{V}_{k}\right)$.
Thus, every set $F\in\mathcal{B}|_{[k,k+1]}$ can be approximated
by sets from $\mathcal{V}_{k}$ and consequently every set $E\in\mathcal{B}$
can be approximated by sets of $\mathcal{V}$. Since \[
\begin{array}{lll}
\1_{C_{\omega,k}}(x) & = & \1_{\sigma^{-n}\left(C+\sum_{i=0}^{n-1}a_{i}N^{i}+N^{n}k\right)}(x)\\
 & = & \1_{C}\left(\sigma^{n}(x)-\sum_{i=0}^{n-1}a_{i}N^{i}-N^{n}k\right)\\
 & = & c\cdot U^{-n}T^{l}\varphi(x),\end{array}\]
 where $l=\sum_{i=0}^{n-1}a_{i}N^{i}+N^{n}k$, we find that $\1_{E}$,
$E\in\mathcal{B}$, can be approximated by linear combinations of
$U^{n}T^{k}\varphi$. Now the claim follows since the simple functions
are dense in $L^{2}\left(H\right)$. 
\end{proof}
For the construction of the mother wavelets we will introduce further
filter functions. Let $A=\{a_{0},\dots,a_{p-1}\}$ and $G:=\{0,\dots N-1\}\backslash\{a_{0},\dots,a_{p-1}\}=\left\{ d_{i}:i=0,\dots,N-p-1\right\} $.
Then the first $N-p$ high pass filters, $m_{1},\dots,m_{N-p}$, on
$\mathbb{T}$ are defined by \[
m_{i+1}:\, z\mapsto z^{d_{i}},\, i\in\underline{N-p-1}.\]
 The remaining $p-1$ filter functions are defined by \[
m_{N-p+k}:\, z\mapsto\frac{1}{\sqrt{p}}\sum_{j=0}^{p-1}\eta^{kj}z^{a_{j}},\:\mbox{for}\: k\in\left\{ 1,\dots,p-1\right\} ,\:\eta=e^{2\pi i/p}.\]

It has been shown in \cite{waveletsonfractals} that the matrix \[
M(z):=\frac{1}{\sqrt{N}}\left(m_{j}(\rho^{l}z)\right)_{j,l=0}^{N-1},\]
 where $\rho=e^{2\pi i/N}$, is unitary for almost all $z\in\mathbb{T}$
(i.e. $M(z)^{*}M(z)=I$, where $I$ denotes the identity matrix). 

The following proposition shows that \[
\left\{ \psi_{i}:=U^{-1}m_{i}\left(T\right)\varphi,\; i\in\{1,\dots,N-1\}\right\} \]
defines a set of mother wavelets. 
\begin{prop}
\label{pro:motherW}The set \[
\{U^{n}T^{k}\psi_{i}:\, i\in\{1,\dots,N-1\},\, n,\, k\in\mathbb{Z}\}\]
 is an ONB for $L^{2}(H)$. \end{prop}
\begin{proof}
In Theorem \ref{mra} it has been shown that the father wavelet $\varphi$
gives rise to an MRA and consequently $\left\{ U^{n}T^{k}\varphi:\, k,n\in\Z\right\} $
spans $L^{2}(H)$. This implies that also $\{U^{n}T^{k}\psi_{i}:\, i\in\{1,\dots,N-1\},\, n,\, k\in\mathbb{Z}\}$
spans $L^{2}\left(H\right)$. Furthermore, the orthonormality of $\left\{ \psi_{i}:\, i\in\{1,\dots,N-1\}\right\} $
follows form the unitarity of the filter functions. Finally, since
$\left\{ T^{k}\psi_{i}:\, i\in\{1,\dots,N-1\},\, k\in\Z\right\} $
is an ONB for $V_{-1}\ominus V_{0}$ and hence $\left\{ U^{n}T^{k}\psi_{i}:\, i\in\{1,\dots,N-1\},\, k\in\Z\right\} $
is an ONB for $V_{n-1}\ominus V_{n}$ the orthonormality of $\{U^{n}T^{k}\psi_{i}:\, i\in\{1,\dots,N-1\},\, n,\, k\in\mathbb{Z}\}$
follows. 
\end{proof}

\section{Fourier Basis\label{sec:Fourier-Basis}}

Also in this section we will use the set up of Section \ref{homeophi}.
That is, we consider an arbitrary IFS $\mathcal{S=}\left(\sigma_{i},i\in\underline{p}\right)$
extended to a {}``gap filling'' IFS $\mathcal{T}=\left(\tau_{i},i\in\underline{N}\right)$
consisting of $N$ contractions such that there exists a set $A\subset\underline{N}$
with $\left\{ \sigma_{i}:i\in\underline{p}\right\} =\left\{ \tau_{j}:j\in A\right\} $.
Additionally, we consider two corresponding \emph{homogeneous linear
IFSs $\widetilde{S}=\left(\widetilde{\sigma}_{i},i\in\underline{p}\right)$
and} $\widetilde{\mathcal{T}}=\left(\widetilde{\tau}_{i},i\in\underline{N}\right)$,
given by the functions $\widetilde{\tau}_{i}:=\rho_{i,N}:x\mapsto\frac{x}{N}+\frac{i}{N}$
such that $\left\{ \widetilde{\sigma}_{i}:i\in\underline{p}\right\} =\left\{ \widetilde{\tau}_{j}:j\in A\right\} $.

\subsection{Construction of a conjugating homeomorphism \label{homemomorphism}}

Let us now investigate the construction of the conjugating homeomorphism
from the linear enlarged fractal to the non-linear one. First the
construction is given for $[0,1]$ to be extended to $\R$ in the
second step. This homeomorphism on $\R$ can be employed for a different
approach to construct the wavelet basis from Section \ref{wavelet}.
See Remark \ref{rem:wavelet homeo} for a detailed discussion.

The aim is to find a homeomorphism $\phi:[0,1]\rightarrow[0,1]$ such
that $\phi(\widetilde{R}_{[0,1]})=R_{[0,1]}$, $\phi(\widetilde{C})=C$
and $\phi\circ\widetilde{\tau}_{i}=\tau_{i}\circ\phi$ for $i\in\underline{N}$,
where $C,$ $\widetilde{C}$ are the limit sets corresponding to the
IFS $\mathcal{S}$, $\widetilde{S}$ respectively, and $R_{[0,1]}$,
$\widetilde{R}_{[0,1]}$ are the corresponding enlarged fractals restricted
to $[0,1]$. The idea of the construction can be found e.g. in \cite{kesse}.

Let $D:=\{f\in C([0,1]):\, f(0)=0,\, f(1)=1,\, f:[0,1]\rightarrow[0,1]\}$
and let the operator $F:D\to D$ be given by \[
(Ff)(x)=\sum_{i=0}^{N-1}\tau_{i}\circ f\circ\widetilde{\tau}_{i}^{-1}(x)\cdot\1_{[\widetilde{\tau}_{i}(0),\widetilde{\tau}_{i}(1))}(x)+\1_{\{1\}}(x),\,\, x\in[0,1].\]

Then it is easy to see that $F$ is a contraction and since $D$ is
complete, we have by the Banach Fixed Point Theorem, that there exists
a fixed point $\phi$ of $F$ in $D$. It is not hard to see that
the inverse function $\phi^{-1}$ of $\phi$ is the unique fixed point
of the contractive operator on $D$ given by \[
(Gh)(x)=\sum_{i=0}^{N-1}\widetilde{\tau}_{i}\circ h\circ\tau_{i}^{-1}(x)\cdot\1_{\left[\tau_{i}(0),\tau_{i}(1)\right)}(x)+\1_{\{1\}}(x),\,\, h\in D,\, x\in[0,1].\]

Consequently $\phi:[0,1]\rightarrow[0,1]$ is a homeomorphism and
it is straight forward to observe that it has all the desired properties.
This homeomorphism may then be extended continuously to $\R$, such
that $\phi(\widetilde{R})=R$. For this notice that any $x\in\mathbb{R}$
can be written uniquely as $x=\{x\}+\left\lfloor x\right\rfloor $,
where $\left\lfloor x\right\rfloor \in\mathbb{Z}$ denotes the largest
integer not exceeding $x$ and $\{x\}=x-\left\lfloor x\right\rfloor $
the fractional part of $x$. Then the extended homeomorphism is defined,
for $x\in\R$, by \[
\widetilde{\phi}(x):=\phi(\{x\})+\left\lfloor x\right\rfloor ,\]
 and consequently, its inverse function $\widetilde{\phi}^{-1}:\mathbb{R}\rightarrow\mathbb{R}$
is given by \[
\widetilde{\phi}^{-1}(z)=\phi^{-1}(\{z\})+\left\lfloor z\right\rfloor ,\; z\in\R.\]

\begin{rem}
\label{rem:wavelet homeo}(1) We would like to remark that the wavelet
bases (constructed in Section \ref{wavelet}) can also be obtained
using the homeomorphisms $\phi:\R\rightarrow\R$. In fact, the wavelet
basis for the non-linear IFS is just the composition of $\phi$ with
the basis elements of the linear IFS constructed in \cite{waveletsonfractals}. 

(2) For what follows we only need the homeomorphism restricted to
$\widetilde{C}$, i.e. $\phi|_{\widetilde{C}}$. Hence, the restricted
homeomorphism is in fact independent of the functions defined on the
{}``gaps'' of the fractal, i.e. depends only on $\left(\widetilde{\tau}_{a}\right)_{a\in A}$.
\end{rem}

\subsection{The appropriate function space}

It will be essential to construct first the Fourier basis for the
linear Cantor set $\widetilde{C}$ given as the limit set of the IFS
$\widetilde{S}=\left(\widetilde{\tau}_{i}:i\in A\right)$. The Hausdorff
dimension of this set $\widetilde{C}$ is $s=\log p/\log N$. The
restriction of the Hausdorff measure $H^{s}$ to $\widetilde{C}$
will be denoted by $\widetilde{\mu}$, i.e. $\widetilde{\mu}=H^{s}|_{\widetilde{C}}$.
This measure satisfies \[
\widetilde{\mu}=\frac{1}{p}\sum_{i\in A}\widetilde{\mu}\circ\widetilde{\tau}_{i}^{-1}\]
 or equivalently, \[
\int f(x)d\mu(x)=\frac{1}{p}\sum_{i\in A}\int f(\widetilde{\tau}_{i}(x))\, d\widetilde{\mu}(x),\: f\in L^{2}(\widetilde{\mu}),\]
 and hence, it is unique with this property by a theorem of Hutchinson
in \cite{Hutchinson}. Furthermore, $H^{s}(\widetilde{C})=\widetilde{\mu}(\widetilde{C})=1$.

We will consider the homeomorphism $\phi:\widetilde{C}\rightarrow C$
from Section \ref{homemomorphism}, where $C$ is the limit set of
$\mathcal{S}=\left(\sigma_{i},i\in\underline{p}\right)=\left(\tau_{i},i\in A\right)$.
Notice that the homeomorphism $\phi$ is measurable with respect to
the Borel-$\sigma$-algebra. This allows us to consider the space
$L^{2}(\mu)$, where $\mu$ is the transported measure, i.e. $\mu=\widetilde{\mu}\circ\phi^{-1}$,
which coincides with the unique measure on $C$ satisfying \[
\mu=\frac{1}{p}\sum_{i\in A}\mu\circ\tau_{i}^{-1}.\]

The above defined homeomorphism $\phi$ will be used to carry over
a Fourier basis in $L^{2}(\widetilde{\mu})$ to a generalised Fourier
basis in $L^{2}(\mu)$ in Section \ref{sub:Fourier-basis-on}.

\subsection{Construction of the Fourier basis for homogeneous linear IFSs}

\label{fourierc} In this section we state the main results on Fourier
bases for homogeneous linear IFSs from Jørgensen and Pedersen, and
Jørgensen and Dutkay (\cite{dense}, \cite{waveletsonfractals}) without
proofs. First we consider the construction of the Fourier basis on
the Hilbert space $L^{2}(\widetilde{\mu})$ with the inner product
$\langle f|g\rangle_{\widetilde{\mu}}=\int f(x)\overline{g(x)}\, d\widetilde{\mu}(x)$.
The \emph{classical Fourier basis} is of the form $\{e_{n}:\, n\in M\}$
with $M\subset\mathbb{Z}$ and $e_{n}:\, x\mapsto e^{i2\pi nx}$.
This kind of basis does not exist for all $L^{2}$-spaces built on
a fractal set, it depends on the underlying algebraic structure of
the IFS (cf. \cite{waveletsonfractals} and Example \ref{exa:-1/3-Cantor-set}).

Recall the definition of the Fourier transform $\widehat{\mu}$ of
a measure $\widetilde{\mu}$: \[
\widehat{\mu}(t):=\int e^{i2\pi tx}\, d\widetilde{\mu}(x),\, t\in\mathbb{R}.\]

\begin{lem}
[\cite{dense}] The Fourier transform for the measure $\widetilde{\mu}$
satisfies the following relation: \[
\widehat{\mu}(t)=\widehat{\mu}(N^{-1}t)\cdot\kappa_{A}(t),\]
 where $\kappa_{A}(t)=\frac{1}{p}\sum_{a\in A}e^{i2\pi tN^{-1}a}$
and $t\in\mathbb{R}$. 
\end{lem}
The following assumption guaranties the existence of a classical Fourier
basis.
\begin{assumption}
\label{ass:UnitaryMatrix} We assume that $0\in A$, and that there
exists a set $L\subset\mathbb{Z}$, such that $0\in L$, $\card L=p$
and \[
H_{AL}:=p^{-1/2}\left(e^{i2\pi\frac{al}{N}}\right)_{a\in A,l\in L}\]
 is unitary. The set $L$ with this property will be called the \emph{dual
set to $A$}.\end{assumption}
\begin{lem}
[\cite{dense}] \label{orthonormal} Under Assumption \ref{ass:UnitaryMatrix}
we have that the set \[
\{e_{\lambda}:\lambda\in\Lambda\}\]
 is orthonormal in $L^{2}(\widetilde{\mu})$, where $e_{\lambda}:\, x\mapsto e^{i2\pi\lambda x}$
and \[
\Lambda:=\{l_{0}+Nl_{1}+N^{2}l_{2}+\dots+N^{k}l_{k}:l_{i}\in L,\; k\in\N\}.\]

\end{lem}
We now introduce three conditions under which the set $\{e_{\lambda}:\,\lambda\in\Lambda\}$
gives an orthonormal basis, i.e. when this set spans the space $L^{2}(\widetilde{\mu})$.
For this we first need the following definitions. Let the dual Ruelle
Operator for the above setting be given by \[
(R_{L}f)(x):=\frac{1}{p}\sum_{l\in L}\left|m_{0}\left(\frac{x-l}{N}\right)\right|^{2}f\left(\frac{x-l}{N}\right),\quad f\in C(\R),\]
 where $m_{0}:\, t\mapsto\frac{1}{\sqrt{p}}\sum_{a\in A}e^{i2\pi ta}$.
\begin{defn}
(\cite{waveletsonfractals})\label{cycle} Let $\widetilde{\sigma}_{b}:=\rho_{-b,N}$
and $L$ and $A:=\left\{ a_{0},\ldots,a_{p-1}\right\} $ are given
as in Assumption \ref{ass:UnitaryMatrix}. Then the the family $\left(z_{1},z_{2},\dots,z_{k}\right)\in\mathbb{T}^{k}$
with $z_{1}=e^{i2\pi\xi_{1}},z_{2}=e^{i2\pi\xi_{2}},\dots,z_{k}=e^{i2\pi\xi_{k}}$
is called an \emph{$L$-cycle with pairing $\left(b_{1},b_{2},\dots,b_{k+1}\right)\in L^{k+1}$,}
if for \emph{$j=1,\ldots,k$} and $z_{k+1}:=z_{1}$ we have \[
z_{j}=\exp\left(i2\pi\sigma_{b_{j}}(\xi_{j+1})\right),\]
 and $\left|m_{0}(\xi_{j})\right|^{2}=p$, \emph{$j=1,\ldots,k$}.\end{defn}
\begin{prop}
[\cite{dense,waveletsonfractals}] \label{pro:-ONB}Under Assumption
\ref{ass:UnitaryMatrix} the follwing three characterisations of the
existence of an orthonormal basis (ONB) in $L^{2}(\widetilde{\mu})$
hold.
\begin{enumerate}
\item The set $\{e_{\lambda}:\lambda\in\Lambda\}$ is an ONB in $L^{2}(\widetilde{\mu})$,
if and only if $Q\equiv1$, where $Q(t):=\sum_{\lambda\in\Lambda}\left|\widehat{\mu}(t-\lambda)\right|^{2}$,
$t\in\R$ .
\item The set $\{e_{\lambda}:\lambda\in\Lambda\}$ is an ONB in $L^{2}(\widetilde{\mu})$,
if the space \[
\left\{ f\in\operatorname{Lip}(\mathbb{R}):f\geq0,\: f(0)=1,\: R_{L}(f)=f\right\} \]
 is one-dimensional.
\item The set $\{e_{\lambda}:\lambda\in\Lambda\}$ is an ONB in $L^{2}(\widetilde{\mu})$,
if the only $L$-cycle is trivial, i.e. is equal to $(1)$. 
\end{enumerate}
\end{prop}

\subsection{Fourier basis on homeomorphic fractals\label{sub:Fourier-basis-on}}

In this section, the above constructed Fourier basis for homogeneous
linear IFSs will be carried over to $L^{2}\left(\mu\right)$. In this
way a generalised Fourier basis is obtained. In fact, the following
proposition shows that the basis elements obtained in our analysis
can again be regarded as characters. Its proof is immediate. 
\begin{prop}
Let $\left(e_{n}\right)$ be the classical Fourier basis on $\R$
and $\phi:\R\rightarrow\R$ be a homeomorphism. Then $d_{n}:=e_{n}\circ\phi^{-1}$
define characters on $(\mathbb{R},\sharp)$, where the addition $\sharp:\mathbb{R}\times\mathbb{R}\rightarrow\mathbb{R}$
is given by $(x,y)\mapsto\phi\left(\phi^{-1}(x)+\phi^{-1}(y)\right)$. 
\end{prop}
Now we are turning to the construction of the generalised Fourier
basis on $L^{2}(\mu)$. It will be crucial that we impose the restriction
$0\in C$. Suppose that $\phi:\widetilde{C}\rightarrow C$ is the
homeomorphism introduced in Section \ref{homemomorphism}. We begin
with the analogue statement to Lemma \ref{orthonormal}.
\begin{lem}
Let $\{e_{\lambda}:\lambda\in\Lambda\}$ (as specified above) be orthonormal
in $L^{2}(\widetilde{\mu})$, then $\{e_{\lambda}\circ\phi^{-1}:\lambda\in\Lambda\}$
is orthonormal in $L^{2}(\mu)$. \end{lem}
\begin{proof}
We have\[
\begin{array}{lll}
\langle e_{\lambda}\circ\phi^{-1}|e_{\lambda'}\circ\phi^{-1}\rangle_{\mu} & = & \int\overline{e_{\lambda}\circ\phi^{-1}}e_{\lambda'}\circ\phi^{-1}\, d\mu\\
 & = & \int e^{-i2\pi\phi^{-1}(t)\lambda}e^{i2\pi\phi^{-1}(t)\lambda'}\, d\mu(t)\\
 & = & \int e^{-i2\pi t\lambda}e^{i2\pi t\lambda'}\, d\widetilde{\mu}(t)\\
 & = & \langle e_{\lambda}|e_{\lambda'}\rangle_{\widetilde{\mu}}\\
 & = & \delta_{\lambda\lambda'}.\end{array}\]
 
\end{proof}
The existence of an ONB can be also transferred by the homeomorphism
$\phi$.
\begin{thm}
\label{thm:homelFourierBases}If $\{e_{\lambda}:\lambda\in\Lambda\}$
is an ONB in $L^{2}(\widetilde{\mu})$, then $\{e_{\lambda}\circ\phi^{-1}:\lambda\in\Lambda\}$
is an ONB in $L^{2}(\mu)$. \end{thm}
\begin{proof}
Only the spanning condition remains to be checked. So let $f\in L^{2}(\mu)\backslash\cl\text{span}\{e_{\lambda}\circ\phi^{-1}:\lambda\in\Lambda\}$,
then $f\circ\phi\in L^{2}(\widetilde{\mu})\backslash\cl\text{span}\{e_{\lambda}:\lambda\in\Lambda\}$.
Hence, $f\circ\phi=0$, since $\{e_{\lambda}:\lambda\in\Lambda\}$
is an ONB of $L^{2}\left(\widetilde{\mu}\right)$. Since $\phi$ bijective,
we also have $f=0$. 
\end{proof}

\section{Examples}

\subsection{Wavelet analysis}
\begin{example}
[$1/3$-Cantor set]\textbf{ \label{exa:-Cantor-set.}} As an example
for the affine case we will determine the wavelet basis for the $1/3$-Cantor
set $C_{3}$ (we refer to \cite{waveletsonfractals} for further details).
The IFS on $[0,1]$ for this set is $\mathcal{S}=\left(\sigma_{k}:k=0,1\right)$
with $\sigma_{k}=\rho_{a_{k},3}:x\mapsto\frac{x+a_{k}}{3}$, $a_{0}:=0$
and $a_{1}:=2$ and the gap filling IFS is $\mathcal{T}=\left(\tau_{k}:k=0,1,2\right)$
with $\tau_{k}:=\rho_{k,3}$. The father wavelet is $\varphi=\chi_{C_{3}}$.
The resulting filter functions on $\mathbb{T}$ are \[
\begin{array}{lll}
m_{0}: & z\mapsto & \frac{1}{\sqrt{2}}(1+z^{2}),\\
m_{1}: & z\mapsto & z,\\
m_{2}: & z\mapsto & \frac{1}{\sqrt{2}}(1-z^{2}).\end{array}\]
 So the mother wavelets are, for $x\in\R$, \[
\begin{array}{lll}
\psi_{1}(x) & = & \sqrt{2}\chi_{C_{3}}(3x-1),\\
\psi_{2}(x) & = & \chi_{C_{3}}(3x)-\chi_{C_{3}}(3x-2).\end{array}\]
 Furthermore, the basis of $L^{2}(H^{s})$, where $s=\log2/\log3$
and $H^{s}$ is the $s$-Hausdorff measure (cf. \cite{waveletsonfractals,Falconer}),
is \[
\left\{ x\mapsto2^{-k/2}\psi_{i}(3^{k}x-l):i=1,2,\, k,l\in\mathbb{Z}\right\} .\]

\end{example}
\begin{example}
[$1/4$-Cantor set with one gap-filling contraction] We now consider
the IFS $\mathcal{S}:=\left(\sigma_{0},\sigma_{1}\right)$ and the
gap filling IFS $\mathcal{T}:=\left(\tau_{0},\tau_{1},\tau_{2}\right)$
with $\sigma_{0}=\tau_{0}=\rho_{0,4},\,\sigma_{1}=\tau_{2}=\rho_{3,4},\,\tau_{1}:x\mapsto\frac{x}{2}+\frac{1}{4}$.
Then the limit set of $\mathcal{S}$ is the $1/4$-Cantor set $C_{4}$
and let $H$ be the fractal measure constructed in Subsection \ref{kolmo}.
Already this example is not covered by \cite{waveletsonfractals}
since the system is not homogeneous, i.e. $\tau_{1}$ has a different
scaling than $\tau_{0}$ and $\tau_{3}$ .
\end{example}
The operators $T$ and $U$ for $f\in L^{2}(H)$ are then given by
$(Tf)(x):=f(x-1)$ and $(Uf)(x):=\frac{1}{\sqrt{2}}f\left(\sigma^{-1}(x)\right)$,
where the scaling function $\sigma$ restricted to $[0,1]$ is given
by \[
\sigma(x):=\mathbbm{1}_{[0,\frac{1}{4})}(x)\cdot4x+\mathbbm{1}_{[\frac{1}{4},\frac{3}{4})}(x)\cdot\left(2x+\frac{1}{2}\right)+\mathbbm1_{[\frac{3}{4},1)}(x)\cdot\left(4x-1\right),\, x\in[0,1].\]
 The father wavelet is $\varphi=\chi_{C_{4}}$. The filter functions
on $\mathbb{T}$ for the construction of the mother wavelets are the
same as for the $1/3$-Cantor case (see Example \ref{exa:-Cantor-set.}),
because the form of the filter functions depends only on the number
and position of the gaps, i.e. \[
\begin{array}{lll}
m_{0}(z) & = & \frac{1}{\sqrt{2}}(1+z^{2}),\\
m_{1}(z) & = & z,\\
m_{2}(z) & = & \frac{1}{\sqrt{2}}(1-z^{2}).\end{array}\]
 So the mother wavelets are given, for $x\in[0,1]$, by \[
\begin{array}{lll}
\psi_{1}(x) & = & (U^{-1}m_{1}(T))\varphi(x)=\sqrt{2}\varphi\left(\sigma(x)-1\right)\\
 & = & \sqrt{2}\varphi\left(\1(x)\cdot4x+\1_{[\frac{1}{4},\frac{3}{4})}(x)\cdot\left(2x+\frac{1}{2}\right)+\1_{[\frac{3}{4},1)}(x)\cdot\left(4x-1\right)-1\right),\\
\psi_{2}(x) & = & (U^{-1}m_{2}(T))\varphi=\varphi\left(\sigma(x)\right)-\varphi\left(\sigma(x)-2\right)\\
 & = & \varphi\left(\1_{[0,\frac{1}{4})}(x)\cdot4x+\1_{[\frac{1}{4},\frac{3}{4})}(x)\cdot\left(2x+\frac{1}{2}\right)+\1_{[\frac{3}{4},1)}(x)\cdot\left(4x-1\right)\right)\\
 &  & -\varphi\left(\1_{[0,\frac{1}{4})}(x)\cdot4x+\1_{[\frac{1}{4},\frac{3}{4})}(x)\cdot\left(2x+\frac{1}{2}\right)+\1_{[\frac{3}{4},1)}(x)\cdot\left(4x-1\right)-2\right).\end{array}\]

Thus, the orthonormal basis for $L^{2}(H)$ is \[
\left\{ U^{n}T^{k}\psi_{i}:\, i=1,2,\, n,k\in\mathbb{Z}\right\} .\]

\subsection{Fourier bases}
\begin{example}
[$1/4$-Cantor set (\cite{analysis})]\textbf{ }Let us recall the standard
example for a Fourier basis for the $1/4$-Cantor set $C_{4}$ supporting
the Cantor measure $\mu_{4}$ and with Hausdorff dimension equal to
$1/2$. This set is the limit set of the IFS on $[0,1]$ $\mathcal{S}=\left(\sigma_{0},\sigma_{1}\right)$
given by$\sigma_{0}=\rho_{0,4}$ and $\sigma_{1}=\rho_{3,4}$. Hence,
in Assumption \ref{ass:UnitaryMatrix} we have $A=\{0,3\}$. For $L:=\{0,2\}$
the matrix $H_{AL}=2^{-1/2}\left(e^{i2\pi al/4}\right)_{a\in A,l\in L}$
is unitary and so the set $\left\{ e_{\lambda}:\lambda\in\Lambda\right\} $
is orthonormal in $L^{2}(\mu_{4})$, where \[
\Lambda=\left\{ \sum_{j=0}^{k}l_{j}4^{j}:\, l_{j}\in\{0,2\},\, k\in\N\right\} .\]
 To show now that $\{e_{\lambda}:\,\lambda\in\Lambda\}$ is an ONB
we will use the characterization by $L$-cycles as stated in Proposition
\ref{pro:-ONB}. We have that $z_{\ell}=e^{i2\pi\xi_{\ell}}\in\mathbb{T}$,
$\ell=1,\ldots,k+1$, is an $L$-cycle of length $k+1$ for $b_{1},\dots,b_{k+1}\in\{0,2\}$
if $z_{j}=e^{i2\pi\frac{\xi_{j+1}-b_{j}}{4}}$, $j=1,\ldots,k$, and
$z_{k+1}=z_{1}$. Thus, $z_{j}=e^{i2\pi\frac{\xi_{j+1}}{4}}$ or $z_{j}=e^{i2\pi\frac{\xi_{j+1}-2}{4}}$
for $j=1,\ldots,k$ and $z_{1}=z_{k+1}$. These conditions can only
be satisfied for $k=0$, i.e. for the cycle $\left(1\right)$. Hence
by Proposition \ref{pro:-ONB}, $\{e_{\lambda}:\lambda\in\Lambda\}$
gives an ONB basis in $L^{2}(\mu_{4})$. 
\end{example}
\begin{example}
[$1/3$-Cantor set]\textbf{ \label{exa:-1/3-Cantor-set}} The $1/3$-Cantor
set is the example for the case where a Fourier basis in the sense
of \cite{waveletsonfractals} does not exists. The $1/3$-Cantor set
$C_{3}$ is given by the  IFS $\mathcal{S}=\left(\sigma_{0},\sigma_{1}\right)$
acting on $[0,1]$ with $\sigma_{0}=\rho_{0,3}$ and $\sigma_{1}=\rho_{2,3}$.
Consequently, in Assumption \ref{ass:UnitaryMatrix} we have $A=\{0,2\}$.
To get a Fourier basis, for the orthonormality a set $L\subset\mathbb{Z}$
is sufficient\underbar{ }such that $\card L=2$ and $H_{AL}$ is unitary.
But it is not possible to find such a set $L$ satisfying these conditions
(cf. \cite{analysis,dense}). If we would relax the condition $L\subset\Z$,
we could choose $L=\{0,\frac{3}{4}\}$ to obtain $H_{AL}$ unitary.
If we now set $\Lambda:=\left\{ \frac{3}{4}(l_{0}+3l_{1}+3^{2}l_{2}+\dots+3^{k}l_{k}):\: l_{i}\in\{0,1\},\, k\in\N\right\} $
we will find that $\{e_{\lambda}:\lambda\in\Lambda\}$ is not orthonormal.
In fact, if we consider $\lambda=\frac{3}{4}$ and $\lambda'=\frac{9}{4}\in\Lambda$,
then \[
\begin{array}{lll}
\langle e_{\lambda}|e_{\lambda'}\rangle_{\mu} & = & \int e^{i2\pi x\left(\frac{9}{4}-\frac{3}{4}\right)}d\mu(x)=\widehat{\mu}\left(\frac{2}{3}\right).\end{array}\]
Since (cf. \cite{dense})\[
\begin{array}{lll}
\widehat{\mu}(t) & = & \int e^{i2\pi tx}\: d\mu(x)\\
 & = & \frac{1}{2}\left(\int e^{i2\pi\frac{t}{3}x}\: d\mu(x)+\int e^{i2\pi\frac{t}{3}x}\cdot e^{i2\pi t\frac{2}{3}}\: d\mu(x)\right)\\
 & = & \frac{1}{2}\left(\widehat{\mu}(\frac{t}{3})+e^{i\frac{4}{3}\pi t}\widehat{\mu}\left(\frac{t}{3}\right)\right)\\
 & = & \frac{1}{2}\left(1+e^{i\frac{4}{3}\pi t}\right)\cdot\widehat{\mu}\left(\frac{t}{3}\right).\end{array}\]
we find\[
\widehat{\mu}\left(\frac{3}{2}\right)=\underbrace{\frac{1}{2}\left(1+e^{i2\pi}\right)}_{=1}\cdot\widehat{\mu}\left(\frac{1}{2}\right)\neq0,\]

\end{example}
This shows that the condition $L\subset\mathbb{Z}$ cannot be omitted.
In fact by \cite{dense,JorgensenPedersen} there does not exist a
set $L\subset\R$ such that $\left\{ e_{\lambda}:\lambda\in\Lambda\right\} $
is a Fourier basis. 
\begin{rem}
(1) From \cite{analysis} we know that there are no more than two
orthogonal functions $e_{\lambda}$ (for any $\lambda\in\mathbb{R}$)
in the Hilbert space $L^{2}(\mu)$. 

(2) It has been shown in \cite{JorgensenPedersen} that for an IFS
with two branches of the form $\sigma_{i}(x)=N^{-1}x+b_{i}$, with
$b_{i}\in\{0,a\}$, $i=1,\,2$, $a\in\R\setminus\{0\}$ and $N\in\Z\setminus\left\{ -1,0,1\right\} $
such that the OSC is satisfied there does not exist a Fourier basis
for any $a\in\R\setminus\{0\}$ if $N$ is odd and there exists a
basis for all $a\in\R\setminus\{0\}$ if $N$ is even and $|N|\geq4$.\end{rem}
\begin{example}
[A generalised Fourier basis on the $1/3$-Cantor set]\textbf{\label{exa:A-generalised-Fourier}}
As seen in the last example, it is not possible to construct a classical
Fourier basis in the sense of \cite{analysis} on $L^{2}(\mu_{3})$,
where $\mu_{3}$ is the Cantor measure on the $1/3$-Cantor set $C_{3}$
given by $\mathcal{S}=\left(\sigma_{0},\sigma_{1}\right)$, $\sigma_{0}=\rho_{0,3}$,
$\sigma_{1}=\rho_{2,3}$. In Section \ref{homemomorphism} we have
shown that there exists a homeomorphism $\phi$ conjugating the IFS
$\mathcal{T}=\left(\tau_{k},k\in\left\{ 0,1,2\right\} \right)$ with
$\tau_{k}=\rho_{k,3}$ and $\widetilde{\mathcal{T}}=\left(\widetilde{\tau}_{0},\widetilde{\tau}_{1},\widetilde{\tau}_{2}\right)$
with $\widetilde{\tau}_{0}=\rho_{0,4}$, $\widetilde{\tau}_{1}:x\mapsto\frac{2x+1}{4}$,
$\widetilde{\tau}_{2}=\rho_{3,4}$ (see Fig. \ref{fig:Homeomorphism}).
\begin{figure}
\center\includegraphics[width=0.45\textwidth]{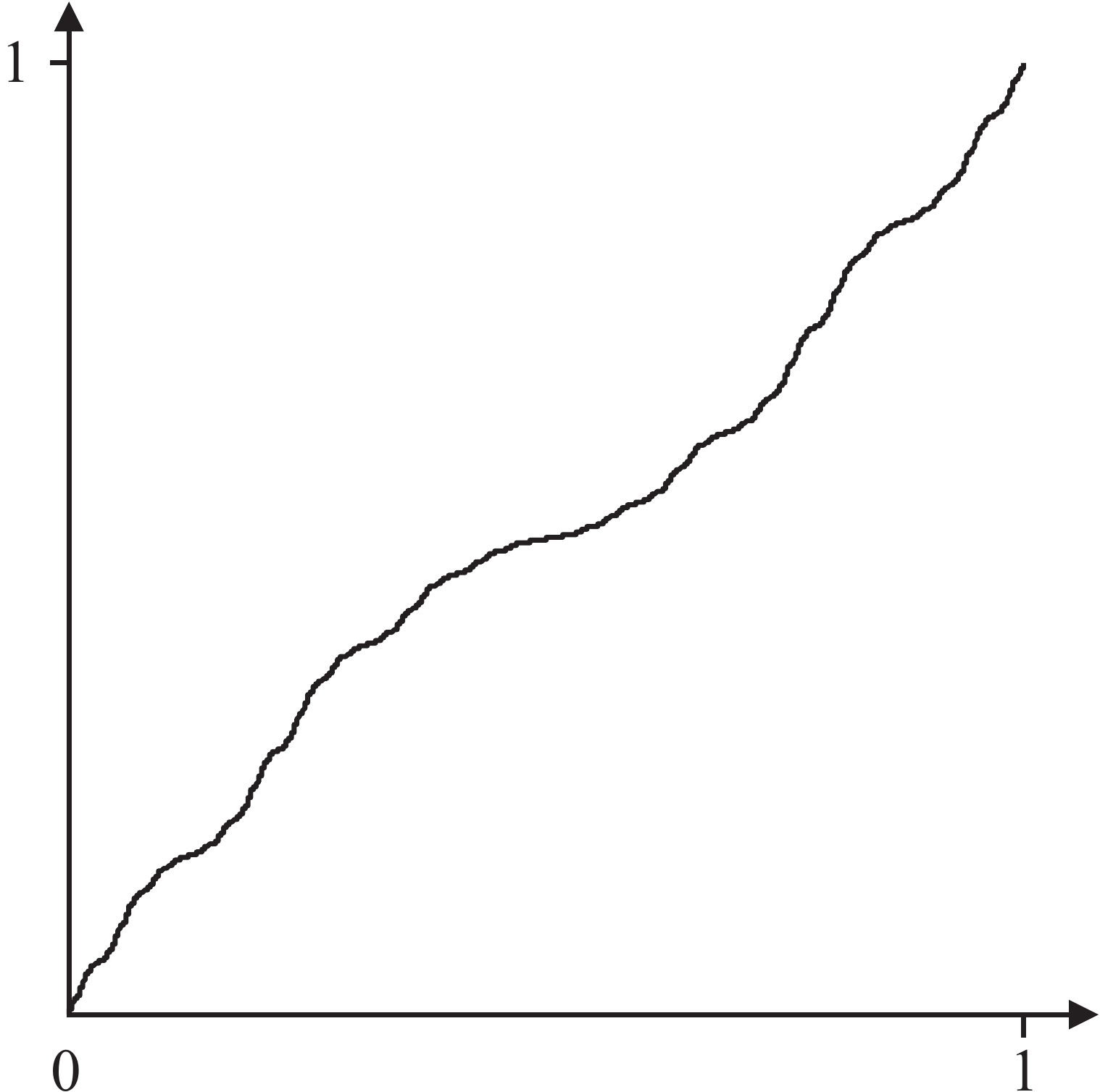}\caption{\label{fig:Homeomorphism}Homeomorphism $\phi:\left[0,1\right]\to\left[0,1\right]$
conjugating the IFSs $\left(\tau_{0},\tau_{1},\tau_{2}\right)$ and
$\left(\widetilde{\tau}_{0},\widetilde{\tau}_{1},\widetilde{\tau}_{2}\right)$
from Example \ref{exa:A-generalised-Fourier}.}

\end{figure}
 Note that $\mu_{3}=\mu_{4}\circ\phi^{-1}$ and that the homeomorphism
restricted to the $1/4$-Cantor set $C_{4}$ is given explicitly by
\[
\begin{array}{lllll}
\phi: & C_{4} & \rightarrow & C_{3}\\
 & \sum_{i}\frac{a_{i}}{4^{i}} & \mapsto & \sum_{i}\frac{\frac{2}{3}a_{i}}{3^{i}}, & a_{i}\in\{0,3\}.\end{array}\]
 Consequently, by Theorem \ref{thm:homelFourierBases} the Fourier
basis of $L^{2}\left(\mu_{4}\right)$ can be carried over to $L^{2}\left(\mu_{3}\right)$.
As mentioned above $\{e_{\lambda}:\lambda\in\Lambda\}$ with $\Lambda=\{\sum_{j=0}^{k}l_{j}4^{j}:\: l_{j}\in\{0,2\},\, k\in\N\}$
is an ONB in $L^{2}(\mu_{4})$ and hence $\{e_{\lambda}\circ\phi^{-1}:\lambda\in\Lambda\}$
is an ONB in $L^{2}(\mu_{3})$ . \end{example}


\begin{thebibliography}{JP98b}
\bibitem[DMP08]{anwendung} J. D'Andrea, K.D. Merrill, J. Packer,
\emph{Fractal wavelets of Dutkay-Jorgensen type for the Sierpinski
gasket space}, Frames and operator theory in analysis and signal processing,
69-88, Contemp. Math., 451, Amer. Math. Soc., Providence, RI, 2008.

\bibitem[Dau92]{daubechies} I. Daubechies, \emph{Ten Lectures on
Wavelets}, CBMS-NSF Regional Conf. Ser. in Appl. Math., vol. 61, SIAM,
Philadelphia, 1992.

\bibitem[DJ06]{waveletsonfractals} D. Dutkay, P. Jørgensen, \emph{Wavelets
on Fractals}, Rev. Math. Iberoamericana, 2006.

\bibitem[Dut06]{positivemaps}D. Dutkay, \textit{Positive definite
maps, representations and frames}, Reviews in Mathematical Physics,
Vol. 16, Issue 04, 2004.

\bibitem[Fal81]{Falconer} K. Falconer, \emph{Fractal Geometry, Mathematical
Foundations and Applications}, John Wily \& Sons, Chichester, 1990.

\bibitem[Hut81]{Hutchinson} J. Hutchinson, \emph{Fractals and self-similarity},
Indiana Univ. Math. J. 30, 1981.

\bibitem[JKPS]{kesse} T. Jordan, M. Kesseböhmer, M. Pollicott, B.
Stratmann, \emph{Sets of non-differentiability for conjugacies between
expanding interval maps}, to appear in: Fundamenta Mathematicae.

\bibitem[Jør06]{analysis} P. Jørgensen, \emph{Analysis and Probability;
Wavelets, Signals, Fractals}, Springer, New York, 2006.

\bibitem[JP98a]{JorgensenPedersen} P. Jørgensen, S. Pedersen, \emph{Orthogonal
harmonic analysis and scaling of fractal measures}, C. R. Acad. Sci.
Paris Sr. I Math. 326, no. 3, p. 301-306, 1998.

\bibitem[JP98b]{dense} P. Jørgensen, S. Pedersen, \emph{Dense analytic
subspaces in fractal $L^{2}$- spaces}, J. Anal. Math. 75, 1998.
\end{thebibliography}
\end{document}